\documentclass[preprint, 12pt, english]{elsarticle}
\usepackage{amsmath}
\usepackage{latexsym, amssymb}
\usepackage{amsthm}
\usepackage{geometry}
\newtheorem{thm}{Theorem}[section] 

\newtheorem{cor}[thm]{Corollary}

\newtheorem{lem}[thm]{Lemma}
\newtheorem{prop}[thm]{Proposition}

\newtheorem{rem}[thm]{Remark}

\newcommand\operA[2]{{\if!#2!\operatorname{#1}\else{\operatorname{#1}_{#2}^{\phantom{I}}}\fi}} 
\newcommand\set[1]{\{#1\}}

\newcommand\charac[1]{\mathrm{char}\left(#1\right)}

\newcommand\Cref[1]{{Corollary~\ref{#1}}}%

\def\s{\sigma}

\newcommand{\Trace}[1][]{\if!#1!\operatorname{Tr}\else{\operatorname{Tr}_{#1}^{\phantom{I}}}\fi} 

\long\def\forget#1\forgotten{{}} %

\def\({\left(}
\def\){\right)}

\newif\iffurther
\furtherfalse

\newif\ifXY 
\XYtrue     
\ifXY

\input xy
\input xyidioms.tex
\usepackage{xy}
\xyoption{all} %
\fi 

\usepackage{babel}

\journal{Linear Algebra and its Applications}

\begin{document}

\begin{frontmatter}

\title{On the generalized Clifford algebra of a monic polynomial}

\author[1]{Adam Chapman}
\ead{adam1chapman@yahoo.com}

\author[2]{Jung-Miao Kuo\corref{cor}}
\ead{jmkuo@nchu.edu.tw}

\cortext[cor]{Corresponding author, postal address: Information Science Building R405, 250 Kuo Kuang Road, Taichung 402, Taiwan, tel: 886-4-2286-0133 ext.405, fax: 886-4-2287-3028}

\address[1]{ICTEAM Institute, Universit\'{e} Catholique de Louvain, B-1348 Louvain-la-Neuve, Belgium}
\address[2]{Department of Applied Mathematics, National Chung-Hsing University, Taichung 402, Taiwan}

\begin{abstract}
In this paper we study the generalized Clifford algebra defined by Pappacena of a monic (with respect to the first variable) homogeneous polynomial $\Phi(Z,X_1,\dots,X_n)=Z^d-\sum_{k=1}^d f_k(X_1,\dots,X_n) Z^{d-k}$ of degree $d$ in $n+1$ variables over some field $F$. We completely determine its structure in the following cases: $n=2$ and $d=3$ and either $\charac{F}=3$, $f_1=0$ and $f_2(X_1,X_2)=e X_1 X_2$ for some $e \in F$, or $\charac{F} \neq 3$, $f_1(X_1,X_2)=r X_2$ and $f_2(X_1,X_2)=e X_1 X_2+t X_2^2$ for some $r,t,e \in F$. Except for a few exceptions, this algebra is an Azumaya algebra of rank nine whose center is the coordinate ring of an affine elliptic curve. We also discuss representations of arbitrary generalized Clifford algebras assuming the base field $F$ is algebraically closed of characteristic zero.
\end{abstract}


\end{frontmatter}

\section{Introduction}
Given a homogeneous polynomial $f(X_1,\dots,X_n)$ in $n$ variables of degree $d$ over a field $F$, its Clifford algebra $C_f$ is defined to be the quotient of the free associative $F$-algebra $F\{x_1,\dots,x_n\}$ by the ideal generated by the elements
$(a_1 x_1+\dots+a_n x_n)^d - f(a_1,\dots,a_n)$ for all $a_1,\dots,a_n \in F$.
This definition was introduced by Roby in \cite{Roby}. It generalizes the classical notion of the Clifford algebra of a quadratic space. If the quadratic space is regular, its Clifford algebra is known to be a tensor product of quaternion algebras over $F$ if the dimension is even or over a quadratic extension of $F$ if the dimension is odd (see \cite{Lam} or \cite{BOI}).

The Clifford algebra is invariant under linear change of variables, and therefore it is used for classifying homogeneous polynomials. Another use of the Clifford algebras is in the study of central simple algebras. Given a central simple algebra $A$ that contains an $F$-vector subspace $V$ such that $v^d \in F$ for all $v \in V$, the subalgebra $F[V]$ of $A$ is a representation of the Clifford algebra $C_f$ where $f(X_1,\dots,X_n)=(v_1 X_1+\dots+v_n X_n)^d$ for some fixed basis $v_1,\dots,v_n$ of $V$.  Consequently, the representations of the Clifford algebras, and in particular their simple images, are of great importance.

The case of $n=2$ and $d=3$ was first considered
by Heerema in \cite{Heerema}. Haile studied these algebras in the series of papers
\cite{Haile, H2, H3, Haile2}. He showed that in characteristic
not $2$ or $3$, if the binary cubic form $f$ is nondegenerate, $C_f$ is an Azumaya algebra of rank 9, whose center is isomorphic to the coordinate ring of the affine
elliptic curve $s^2=r^3-27\Delta$ where $\Delta$ is the
discriminant of $f$. Hence the simple homomorphic
images of $C_f$ are cyclic algebras of degree $3$. Moreover, there is a one-to-one correspondence between Galois orbits of points on that affine elliptic curve and the simple homomorphic images of $C_f$. The induced mapping from the group of $F$-rational points on the elliptic curve described above with the point at infinity as unity to the Brauer group of $F$ is a group homomorphism with image the relative Brauer group of the function field of the curve $C$ given by $Z^3-f(X,Y)=0$. Finally, he proved that $C_f$ is split if and only if the curve $C$ has an $F$-rational point.

On the other hand, irreducible representations of Clifford algebras were studied in papers such as \cite{CKM, HT, Van}. It was shown in \cite{HT} that if $f$ is nondegenerate and $F$ is an infinite field, the dimension of any representation of $C_f$ is divisible by the degree of $f$. In the case where $F$ is algebraically closed of characteristic 0, Van den Bergh in \cite{Van} showed that the equivalence classes of $dr$-dimensional representations of $C_f$ are in one-to-one correspondence with isomorphism classes of certain vector bundles (later known as Ulrich bundles) of rank $r$ on the hypersurface in $\mathbb{P}^n$ given by the equation $Z^d- f(X_1,\dots,X_n)=0$. Advanced results on representations of Clifford algebras were recently derived in \cite{CKM} via the study of the Ulrich bundles on the geometric side of this correspondence.

In \cite{Pappacena}, Pappacena generalized the notion of the Clifford algebra to the algebra associated to a monic polynomial (with respect to the first variable) over $F$ of the form $\Phi(Z,X_1,\dots,X_n)=Z^d-\sum_{k=1}^d f_k(X_1,\dots,X_n) Z^{d-k}$ where each $f_k$ is a homogeneous polynomial of degree $k$.
This algebra, denoted there by $C_\Phi$, is defined to be the quotient of the free associative $F$-algebra $F\{ x_1,\dots,x_n\}$ by the ideal generated by the elements
\begin{equation*}
(a_1 x_1+\dots+a_n x_n)^d -\sum_{k=1}^d f_k(a_1,\dots,a_n) (a_1 x_1+\dots+a_n x_n)^{d-k}
\end{equation*}
for all $a_1,\dots,a_n \in F$.
The algebra $C_{\Phi}$ is called the generalized Clifford algebra of the monic polynomial $\Phi$. Pappacena proved in that paper that if $d=2$ then this generalized Clifford algebra is isomorphic to the classical Clifford algebra of a quadratic form, and therefore its structure is known. However, he said little when $d\geq 3$.

In \cite{Kuo}, Kuo 
studied the generalized Clifford algebra of the monic polynomial $\Phi(Z,X,Y)=Z^3-e XYZ-f(X,Y)$ in characteristic not 2 or 3 and derived the results very similar to those obtained by Haile when $\Phi$ is absolutely irreducible. In particular, $C_{\Phi}$ is an Azumaya algebra of rank 9 and its center is isomorphic to the coordinate ring of an affine elliptic curve whose closure in $\mathbb{P}^2$ is the Jacobian of the cubic curve given by $\Phi(Z,X,Y)=0$. This result was recently applied in \cite{FN} to solving a problem about $3\times 3\times 3$ cubes, which relates to the computation of the Cassels-Tate pairing on the 3-Selmer group of an elliptic curve over a number field. These generalized Clifford algebras with $e\neq 0$ were also used in \cite{HK}, where the reducible case was fully understood, to classify hyperbolic planes in the space of reduced trace zero elements in a division algebra of degree three into three basic types.

The contents of this paper are structured as follows. In Section \ref{evd} we present the basic tool, namely the eigenvector decomposition, which enables us to obtain the results in the following sections.
In Section \ref{Seccharnot3} we study the generalized Clifford algebra of the monic polynomial $\Phi(Z,X,Y)=Z^3-r YZ^2-(e XY+t Y^2) Z-(\alpha X^3+\beta X^2 Y+\gamma X Y^2+\delta Y^3)$ in characteristic not $3$.
The formulas for the simple homomorphic images provided in this section fill the gap that was left open in \cite{Kuo}, where formulas were provided only in the case of $r=t=\beta=\gamma=0$.
In Section \ref{Secchar3}, we study the generalized Clifford algebra of the monic polynomial $\Phi(Z,X,Y)=Z^3-e XYZ-f(X,Y)$ in characteristic $3$. This complements the results in \cite{Kuo}, which were derived only for characteristic different from $3$.
In Section \ref{Secchar0}, we study representations of arbitrary generalized Clifford algebras assuming $F$ is algebraically closed of characteristic zero. We shall give positive answers to both parts of Question 5.2 in \cite{Pappacena}.

\section{Eigenvector Decomposition}\label{evd}
Let $d$ be a positive integer. In an associative algebra $A$ over a field $F$, a non-central element $x \in A$ satisfying $x^d \in F$ and $x^k \not \in F$ for any $1\leq k < d$ is called $d$-central. The main tool we will be using in Section 3 to analyze the structure of the generalized Clifford algebra is the eigenvector decomposition of elements with respect to conjugation by a $d$-central element. The following lemma is taken from \cite{Chapman}.

\begin{lem}\label{eigencharnot}
In an associative algebra $A$ over a field $F$ of characteristic prime to $d$ containing a primitive $d$th root of unity $\rho$, if $x$ is a $d$-central element with $x^d\neq {0}$ then for every $y \in A$, there exist elements $y_0, y_1, \dots, y_{d-1}$ such that $y=y_0+ y_1+\dots+y_{d-1}$ and $y_k x=\rho^k x y_k$.
\end{lem}

\begin{proof}
Take $y_k=\frac{1}{d}(y+\rho^k x y x^{-1}+\dots+\rho^{k (d-1)} x^{d-1} y x^{1-d})$ for $0 \leq k \leq d-1$. It is easy to see that $y=y_0+ y_1+\dots+y_{d-1}$ and $y_k x=\rho^k x y_k$.
\end{proof}

In case $F$ is of characteristic prime to $d$ and contains a primitive $d$th root of unity $\rho$, $d$-central elements play an important role in cyclic algebras of degree $d$ over $F$.
Every cyclic algebra over $F$ of degree $d$ has a symbol presentation
$$(\alpha,\beta)_{d,F}=F\langle x,y : x^d=\alpha, y^d=\beta, y x y^{-1}=\rho x\rangle$$ for some $\alpha,\beta \in F^\times$.
The elements $x$ and $y$ are clearly $d$-central. In fact, for every $d$-central element $x$ in such an algebra there exists another $d$-central element $y$ satisfying $y x y^{-1}=\rho x$, and so the algebra has the symbol presentation $(x^d,y^d)_{d,F}$. Furthermore, if $x,y$ are invertible elements such that $y x y^{-1}=\rho x$, then they are $d$-central and the algebra has the symbol presentation $(x^d,y^d)_{d,F}$.

In characteristic $p> 0$, there are two interesting types of elements, the Artin-Schreier elements, that is, elements that satisfy an equation of the form $x^p-x=\alpha$ for some $\alpha \in F$, and $p$-central elements as defined above. In this case, however, the $p$-central elements generate purely inseparable field extensions over the base field, while the Artin-Schreier elements generate Galois field extensions.

The following two lemmas are the results of decomposition of elements with respect to these two special types of elements, and will be used in Section 4. We write $[\mu,\nu]_1=[\mu,\nu]=\nu \mu-\mu \nu$ and define $[\mu,\nu]_k$ for $k>1$ inductively as $[[\mu,\nu]_{k-1}, \nu]=\nu [\mu,\nu]_{k-1}-[\mu,\nu]_{k-1} \nu$. Also, $[\mu,\nu]_0$ is defined to be $\mu$.

\begin{lem}\label{decomcharp}
In an associative algebra $A$ over a field $F$ of characteristic $p$,
if $x$ is Artin-Schreier then for every $z \in A$, there exist $\set{z_k : k \in \mathbb{Z}_p}$ such that $z=z_0+z_1+\dots+z_{p-1}$ and $[z_k,x]=k z_k$. By taking $t_{p-k}=z_k$, we have $z=t_0+\dots+t_{p-1}$ and $[x,t_k]=k t_k$.
\end{lem}

\begin{proof}
Let $z_0=z-[z,x]_{p-1}$, and for $1 \leq k \leq p-1$, $z_k=-(k^{p-2} [z,x]_1+\dots+k [z,x]_{p-2}+[z,x]_{p-1})$.
An easy calculation shows that $[z_k,x]=k z_k$. It is obvious that $z_0+z_1+\dots+z_{p-1}=z$.
\end{proof}

\begin{lem}\label{decomcharp2}
In an associative algebra $A$ over a field $F$ of characteristic $p$,
if $y$ is $p$-central then for every $z \in A$, there exist $\set{z_k : k \in \mathbb{Z}_p}$ such that $[z_0,y]=0$, $[z_k,y]=z_{k-1}$ for each $k \neq 0$, and $z=z_{p-1}-z_{p-2}$.
\end{lem}

\begin{proof}
For $k \neq 0$, let $z_k=[z,y]_{p-1-k}+[z,y]_{p-k}+\dots+[z,y]_{p-1}$.
Let $z_0=[z,y]_{p-1}$.
It is clear that they satisfy the requirements.
\end{proof}

Artin-Schreier elements and $p$-central elements play an important role in cyclic algebras of degree $p$ over $F$.
Every cyclic algebra over $F$ of degree $p$ has a symbol presentation
$$[\alpha,\beta)_{p,F}=F\langle x,y : x^p-x=\alpha, y^p=\beta, y x y^{-1}=x+1\rangle$$
for some $\alpha,\beta \in F$.
The element $x$ is clearly an Artin-Schreier element and the element $y$ is $p$-central.
In fact, for every Artin-Schreier element $x$ in such an algebra there exists some $p$-central element $y$ satisfying $y x y^{-1}=x+1$, and so the algebra has the symbol presentation $[x^p-x,y^p)_{p,F}$. Similarly, for every $p$-central element $y$ there exists an Artin-Schreier element $x$ such that $y x y^{-1}=x+1$. Furthermore, if there exist two elements $x$ and $y$ satisfying $y x y^{-1}=x+1$, then $x$ is Artin-Schreier, $y$ is $p$-central and the algebra has the symbol presentation $[x^p-x,y^p)_{p,F}$.

\section{Characteristic not three}\label{Seccharnot3}
In this section we study the generalized Clifford algebra of the
monic polynomial $\Phi(Z,X,Y)=Z^3-r Y Z^2-(e XY+t Y^2) Z-(\alpha
X^3+\beta X^2 Y+\gamma X Y^2+\delta Y^3)$ over the field $F$
assuming the characteristic of $F$ is not 3. As mentioned in the
introduction, the case of $r=t=0$ was studied in \cite{Kuo}, where
the simple homomorphic images were described explicitly only in the
diagonal case $\beta=\gamma=0$. Here we will apply the eigenvector
decomposition to provide formulas for the simple homomorphic images
of $C_{\Phi}$.

For the sake of simplicity, we adopt the notation used in \cite{Revoy}:
$x_1^{d_1} \ast \dots \ast x_t^{d_t}$ denotes the
sum of all the products where each $x_i$ appears $d_i$ times. For
example, $x^2*z^2 = xxzz+xzxz+xzzx+zxxz+zxzx+zzxx$. As usual we may
omit exponents $d_i = 1$, so that $x^2*y = xxy+xyx+yxx$. This
notation is commutative in the sense that $x_1^{d_1} \ast \dots
\ast x_t^{d_t} = x_{\s(1)}^{d_{\s(1)}} * \cdots *
x_{\s(t)}^{d_{\s(t)}}$ for any permutation $\s\in S_t$.

The algebra $C_\Phi$ is by definition
\begin{equation*}
\begin{aligned}
C_\Phi=F\langle x,y\colon &  x^3=\alpha,\\
 & x^2 * y=r x^2+e x+\beta,\\
& x * y^2=r x y+r y x+t x+e y+\gamma,\\
 & y^3=r y^2+t y+\delta\rangle.
\end{aligned}
\end{equation*}
We assume that $\alpha\neq 0$ and $F$ contains a primitive third root of unity $\rho$. According to Lemma \ref{eigencharnot}, since $x^3 \in F^{\times}$, there exist $y_0, y_1, y_2\in C_{\Phi}$ such that
\begin{equation}\label{decomp}
y=y_0+y_1+y_2\ \textrm{ and }\ y_i x=\rho^{i} x y_i.
\end{equation} In this case, we say that $y_i$ $\rho^i$-commutes with $x$. Under this decomposition, the relation $x^2 * y=r x^2+e x+\beta$ is equivalent to  $y_0=(3 \alpha)^{-1} (e x^2+\beta x+\alpha r)$.
Substituting this in the relation $x * y^2=r x y+r y x+t x + e y+\gamma$, we obtain by a straight-forward calculation that
\begin{equation}\label{y1y2}
y_1 y_2=\rho y_2 y_1+\frac{(1-\rho)D_1}{3 \alpha} x^2-\frac{(1-\rho)D_2}{9 \alpha}
\end{equation}
where $$D_1=\gamma+\frac{e r}{3} -\frac{\beta^2}{3 \alpha}\ \textrm{ and }\ D_2= e \beta - 3 \alpha t-\alpha r^2.$$
Let $$w=x^{-1} y_2 y_1+\rho^2 \frac{D_1}{3 \alpha} x+\frac{D_2}{9 \alpha} x^{-1}.$$

\begin{lem}\label{central}
The elements $w,y_1^3$ and $y_2^3$ are in the center of $C_\Phi$.
\end{lem}
\begin{proof}
Clearly, $w$ commutes with $x$. By Equation \eqref{y1y2} we see that $y_i w = wy_i$, $i=1, 2$. Thus, $w$ commutes with $y$ and hence is central in $C_\Phi$. Similarly, one can check that $y_1^3$ and $y_2^3$ are central in $C_\Phi$.
\end{proof}

The relation $y^3=r y^2+t y+\delta$ may be split into three parts due to conjugation by $x$. The part on the left-hand side which $\rho$-commutes with $x$ is $y_0^2 * y_1+y_1^2 * y_2+y_2^2 * y_0$ and on the right-hand side it is $r y_0 * y_1+r y_2^2+t y_1$.
A direct computation shows that $y_0^2 * y_1=(3\alpha)^{-1}e\beta y_1 + 3^{-1}r^2 y_1- (3 \alpha)^{-1}\rho e r x^2 y_1-(3 \alpha)^{-1}\rho^2 \beta r x y_1$,   $y_1^2 * y_2=-(3 \alpha)^{-1}D_2 y_1$, $y_2^2 * y_0=r y_2^2$ and $r y_0 * y_1=3^{-1}2 r^2 y_1-(3 \alpha)^{-1}\rho e r x^2 y_1-(3 \alpha)^{-1}\rho^2 \beta r x y_1$. Thus, $y_0^2 * y_1+y_1^2 * y_2+y_2^2 * y_0$ is equal to $r y_0 * y_1+r y_2^2+t y_1$. Similarly, the part on the left-hand side which $\rho^2$-commutes with $x$ is equal to that on the right-hand side: $y_0^2 * y_1+y_1^2 * y_2+y_2^2 *  y_0 = r y_0 * y_2+r y_1^2+ty_2$.
So let us consider the parts on both sides which commute with $x$:
\begin{equation}\label{y^3}
\left(\frac{1}{3 \alpha} (e x^2+\beta x+\alpha r)\right)^3+y_1^3+y_2^3+\frac{1}{3 \alpha} (e x^2+\beta x+\alpha r) * y_1 * y_2=r y_0^2+r y_1 * y_2+t y_0 +\delta.
\end{equation}
We first compute
\begin{equation*}
\begin{aligned}
& (e x^2+\beta x) * y_1 * y_2\\
=\ & e((2+\rho^2)x^2 y_1 y_2 +(2+\rho) x^2 y_2 y_1) + \beta((2+\rho)xy_1 y_2 + (2+\rho ^2)xy_2 y_1)\\
=\ &  e\left(-3\rho^2 x^2 y_2 y_1-\rho D_1 x+\frac{\rho D_2}{3\alpha}x^2\right) + \beta\left(D_1-\frac{D_2}{3 \alpha} x\right) \\
=\ & e\left(-3\rho^2 \alpha w-\frac{D_2}{3 \alpha} x^2\right) +\beta\left(D_1-\frac{D_2}{3 \alpha} x\right)\\
\end{aligned}
\end{equation*}
where the second equality holds by applying Equation \eqref{y1y2}. Substituting this in Equation \eqref{y^3} we then obtain by another straight-forward calculation that
\begin{equation}\label{y1^3y2^3}
D+y_1^3+y_2^3-\rho^2 e w=0,
\end{equation}
where $$D=\frac{e^3}{27 \alpha}+\frac{\beta^3}{27 \alpha^2}-\frac{2 r^3}{27}+\frac{\beta}{3 \alpha} D_1-\frac{r t}{3}-\delta.$$
Consequently, via the decomposition in \eqref{decomp} with $y_0$ taken as $(3 \alpha)^{-1} (e x^2+\beta x+\alpha r)$, $C_{\Phi}$ is an $F$-algebra generated by $x, y_1, y_2$ subject to the relations $x^3=\alpha$, $y_i x=\rho^i xy_i$ and Equations \eqref{y1y2}, \eqref{y1^3y2^3}. Thus we have the following result.

\begin{lem}\label{27fg}
As an  $F[y_1^3, w]$-module, $C_\Phi$ is finitely generated by the 27 elements $x^i y_1^j y_2^k$, where $0\leq i,j,k\leq 2$.
\end{lem}

Let us consider the algebra after the localization $C_\Phi[y_1^{-3}]$.
Since in this algebra $y_1$ is invertible, from the choice of $w$, we have
\begin{equation}\label{y2}
y_2=x y_1^{-1} w-\rho^2 \frac{D_1}{3 \alpha} x^2 y_1^{-1}-\frac{D_2}{9 \alpha} y_1^{-1},
\end{equation}
and so
\begin{equation*}
\begin{aligned}
  & y_2^3 
=\alpha y_1^{-3} w^3-\frac{D_1^3}{27 \alpha} y_1^{-3}-\frac{D_2^3}{729 \alpha^3} y_1^{-3}-\rho^2 \frac{D_1D_2}{9 \alpha} wy_1^{-3}.
\end{aligned}
\end{equation*}
Therefore, substituting this in Equation \eqref{y1^3y2^3} we get
$$D+y_1^3+\alpha y_1^{-3} w^3-\frac{D_1^3}{27 \alpha} y_1^{-3}-\frac{D_2^3}{729 \alpha^3} y_1^{-3}-\rho^2 \frac{D_1D_2}{9 \alpha} wy_1^{-3}-\rho^2 e w=0.$$
Consequently,
\begin{equation}\label{y1^3w}
(D-\rho^2 e w) y_1^3+y_1^6+\alpha w^3-\frac{D_1^3}{27 \alpha}-\frac{D_2^3}{729 \alpha^3}-\rho^2 \frac{D_1D_2}{9 \alpha} w=0.\end{equation}
The last equality also holds in $C_\Phi$.

We next show that the center $Z$ of $C_\Phi$ is $F[y_1^3,w]$ and it is isomorphic to the coordinate ring of the affine elliptic curve
\begin{equation}\label{E}
E\colon (D-\rho^2 e R) S+S^2+\alpha R^3-\frac{D_1^3}{27 \alpha}-\frac{D_2^3}{729 \alpha^3}-\rho^2 \frac{D_1D_2}{9 \alpha} R=0,
\end{equation}
where the discriminant is assumed to be nonzero. Let $E$ also denote the elliptic curve with affine piece given by Equation \eqref{E}.

\begin{prop}\label{isom}
There is an $F$-algebra isomorphism from $C_\Phi[y_1^{-3}]$ into the symbol algebra $(\alpha, S)_{3,F(E)}$ over the function field $F(E)$ of the elliptic curve $E$.
\end{prop}
\begin{proof}
Let $u, v$ be the generators of $(\alpha, S)_{3,F(E)}$ satisfying $u^3=\alpha, v^3=S$ and $vu=\rho uv$.
Let $\phi$ be the $F$-algebra homomorphism from $C_\Phi$ into $(\alpha, S)_{3, F(E)}$ defined as follows
\begin{equation*}
\begin{aligned}
\phi\colon C_\Phi\
 & \rightarrow\  (\alpha, S)_{3,F(E)}\\
x\ & \mapsto\  u \\
y_1\ & \mapsto\  v\\
y_2\ & \mapsto\  u\left(R-\rho^2\frac{D_1}{3\alpha}u-\frac{D_2}{9\alpha}u^{-1}\right)v^{-1}.
\end{aligned}
\end{equation*}
One can check that $x^3=\alpha, y_i x=\rho^i xy_i$ and the relations in Equations \eqref{y1y2} and \eqref{y1^3y2^3} are preserved under the map $\phi$. Thus it is well-defined and $\phi(w)=R$. Furthermore, it induces a homomorphism from $C_\Phi[y_1^{-3}]$ to $(\alpha, S)_{3, F(E)}$, which we also denote by $\phi$.

Notice that from Equations \eqref{y2} and \eqref{y1^3w}, $C_\Phi[y_1^{-3}]$ as an $F[y_1^{\pm 3}]$-module is finitely generated by the 27 elements $x^i y_1^j w^k$, where $0\leq i,j,k\leq 2$. Since the images of these elements are linearly independent over $F[S^{\pm 1}]$ and $\phi$ when restricted to $F[y_1^{\pm 3}]$ is injective, it follows that $\phi$ itself is injective.
\end{proof}

\begin{cor}\label{Z}
The center $Z$ of $C_\Phi$ is $F[y_1^3, w]$ and it is isomorphic to the coordinate ring $F[E]$ of the affine elliptic curve $E$.
\end{cor}
\begin{proof}
By Proposition \ref{isom}, $F[y_1^3, w]\cong F[E]$, a Dedekind domain. Furthermore, we see from its proof that $\phi(C_\Phi)F(E)=(\alpha, S)_{3,F(E)}$. In particular, the center of $\phi(C_\Phi)$ is contained in $F(E)$. Therefore $F[E]=\phi(F[y_1^3, w])\subseteq \phi(Z) \subseteq F(E)$. It follows from Lemma \ref{27fg} and the injectivity of $\phi$ that $Z=F[y_1^3, w]\cong F[E]$.
\end{proof}

Now the center of $C_\Phi[y_1^{-3}]$ is $Z_{(y_1^3)}=F[y_1^{\pm 3}, w]\cong F[E]_{(S)}$ in which $y_1^3$ is invertible. Thus we have the following result.

\begin{cor}\label{L1}
$C_\Phi[y_1^{-3}]$ is the symbol Azumaya algebra $(\alpha, y_1^3)_{3,F[y_1^{\pm 3},w]}$. Similarly, $C_\Phi[y_2^{-3}]=(y_2^3, \alpha)_{3,F[y_2^{\pm 3},w]}$.
\end{cor}

From now on, we restrict ourselves to the following conditions:
$D\neq 0$ and the subalgebra $F[x: x^3=\alpha ]$ is a field.
\begin{prop}\label{eitheror}
In every homomorphic image of $C_\Phi$, either $y_1^3 \neq 0$ or $y_2^3 \neq 0$. In particular, if the image is simple then either $y_1^3$ or $y_2^3$ is invertible.
\end{prop}

\begin{proof}
Assume to the contrary that $y_1^3=y_2^3=0$. Then by Equation \eqref{y1^3y2^3}, $D=\rho^2 ew$. If $e=0$, then $D=0$, a contradiction. If $e\neq 0$, then by the choice of $w$, we have that $y_2 y_1=\rho e^{-1}Dx -(3 \alpha)^{-1}\rho^{2} D_1 x^2 - (9 \alpha)^{-1}D_2$, which is invertible as a nonzero element of the field $F[x]$.
However this means that $y_1$ is invertible too, which is a contradiction.
\end{proof}

\begin{cor}
The algebra $C_\Phi$ is Azumaya of rank 9.
\end{cor}
\begin{proof}
By Corollary \ref{Z} and Lemma \ref{27fg}, $C_\Phi$ is finitely generated as a module over its center $Z=F[y_1^3, w]$. For every maximal ideal $m$ of $Z$, it follows from Proposition \ref{eitheror} and Corollary \ref{L1} that $C_\Phi/mC_\Phi$ is a central simple algebra of degree 3 over the field $Z/m$. Therefore, $C_\Phi$ is Azumaya of rank 9.
\end{proof}

\begin{rem}
Another way to prove that $C_\Phi$ is Azumaya is the following: Every $\Phi(Z,X,Y)=Z^3-\sum_{k=1}^3 f_k(X, Y) Z^{3-k}$ can be linearly transformed over $\bar{F}$ into the one with $f_1=0$,  $f_2=eXY$ and $f_3=X^3 +Y^3$ for some $e\in \bar{F}$ (in characteristic not 2 or 3), and therefore that $C_{\Phi}$ is Azumaya follows immediately from \cite{Kuo} and the fact that the construction of $C_{\Phi}$ is functorial in $F$.
\end{rem}

We finally are able to describe explicitly the simple homomorphic images of $C_\Phi$.

\begin{thm}\label{11corresp}
There is a one-to-one correspondence between the simple homomorphic images of $C_\Phi$ and the Galois orbits of $\bar{F}$-rational points on the affine elliptic curve $E$ as follows: the Galois orbit containing $(R_0,S_0)$ on $E$ gives rise to the $F(R_0,S_0)$-central simple algebra $(\alpha, S_0)_{3,F(R_0,S_0)}$ if $S_0\neq 0$ and $(\rho^2 e R_0-D, \alpha)_{3,F(R_0,S_0)}$ if $S_0=0$.
\end{thm}
\begin{proof}
Since $C_\Phi$ is Azumaya, there is an one-to one correspondence between its simple homomorphic images and maximal ideals of its center $Z\cong F[E]$. Furthermore, $y_1^3, w$ in the center correspond to $S, R$. Thus by Equation \eqref{y1^3y2^3}, $y_2^3$ corresponds to $\rho^2 e R-S-D$. Therefore, the result follows from Proposition \ref{eitheror} and Corollary \ref{L1}.
\end{proof}

Define the function $\Psi$ from the group $E(F)$ of $F$-rational points on the elliptic curve $E$ into the Brauer group of $F$ as follows
\begin{align*}
\Psi\colon E(F)
 & \rightarrow Br(F)\\
(R_0,S_0) & \mapsto \begin{cases}
[(\alpha, S_0)_{3,F}] & \textrm{if}\ S_{0}\neq0\\
[(\rho^2 e R_0-D, \alpha)_{3,F}] & \textrm{if}\ S_{0}=0\end{cases}\\
O & \mapsto 1.
\end{align*}
We next show that the arguments used in \cite[Section 4]{Kuo} can be applied here to show that $\Psi$ is a group homomorphism.

\begin{prop}
The function $\Psi$ is a group homomorphism.
\end{prop}
\begin{proof}
Identify $Z=F[y_1^3, w]$ with $F[E]$. Similar to the proof of \cite[Corollary 4.3]{Kuo}, the Brauer class of $C_\Phi$ in $Br(F(E))$ is unramified everywhere. Thus, the algebra $C_\Phi$ can be extended to a Brauer class in $Br(E)$. Also, $C_\Phi\otimes_{F[E]} F(E)=(\alpha, S)_{3,F(E)}=(\alpha, T)_{3,F(E)}$, where $T = R^3/S^2$. By Equation \eqref{E} we see that
\begin{equation}\label{T}
T = \frac{D-\rho^2 e R}{-\alpha S}-\frac{1}{\alpha}+\frac{D_1^3}{27\alpha^2 S^2}+\frac{D_2^3}{729 \alpha^4 S^2}+\frac{\rho^2 D_1D_2R}{9\alpha^2 S^2}.
\end{equation}
Let $\nu$ be the discrete valuation on $F(E)$ corresponding to $O$. Then $\nu(R)=-1$ and $\nu(S)=-3/2$. Thus  $\nu(T)=0=\nu(\alpha)$, and hence the specialization of $C_\Phi\otimes_{F[E_a]} F(E)$ at $O$ is $(\alpha, \bar{T})_{3,F}$ where $\bar{T}$ is the image of $T$ in the residue field of $O$. By Equation \eqref{T}, $\bar{T}=-1/\alpha= N_{F(\sqrt[3]{\alpha})/F}((-1/\alpha)\sqrt[3]{\alpha^2})$. Thus, the specialization at $O$ of the class of $C_\Phi$ in $Br(E)$ is trivial. Therefore, similar to the proof of \cite[Theorem 4.1]{Kuo}, the result now follows from Lemma 3.2 and Theorem 3.5 of \cite{CK}.
\end{proof}

Since $C_\Phi$ is Azumaya of rank 9, one can check that the homogeneous polynomial $\Phi(X,Y,Z)$ over $F$ is then absolutely irreducible. Let $C$ denote the cubic curve given by the equation $\Phi(X,Y,Z)=0$. The computations in \cite{A} show that the elliptic curve $E$ is the Jacobian of the cubic curve $C$. We have the following two analogues of Proposition 4.5 and Theorem 4.6 of \cite{Kuo} with similar proofs, which we therefore skip.

\begin{prop}\label{F(C)}
The group homomorphism $\Psi\colon E(F)\rightarrow Br(F)$ maps onto the relative Brauer group $Br(F(C)/F)$.
\end{prop}

\begin{prop}
The Azumaya algebra $C_\Phi$ is split if and only if the cubic curve $C$ has an $F$-rational point.
\end{prop}

\section{Characteristic three}\label{Secchar3}

In this section we study the generalized Clifford algebra of the monic polynomial $\Phi(Z,X,Y)=Z^3-e XYZ-(\alpha X^3+\beta X^2 Y+\gamma X Y^2+\delta Y^3)$ for some $e,\alpha,\beta,\gamma,\delta \in F$, assuming that the field $F$ is of characteristic 3. The algebra $C_\Phi$ is by definition $$F\langle x,y\colon   x^3=\alpha, y^3=\delta, x^2 * y=e x+\beta, x * y^2=e y+\gamma\rangle.$$
We assume that $\alpha \neq 0$ and treat the two cases of $e=0$ and $e \neq 0$ separately.

\subsection{$e=0$}

In this case, $C_\Phi$ is simply the ordinary Clifford algebra of the form $f(X,Y)=\alpha X^3+\beta X^2 Y+\gamma X Y^2+\delta Y^3.$
The element $x$ is $3$-central. Therefore, according to Lemma \ref{decomcharp2}, we can decompose $y$ as
$$y=y_2-y_1$$
such that
\begin{equation}\label{y012}
x y_2-y_2 x=y_1, x y_1-y_1 x=y_0, \textrm{ where }y_0 x=x y_0.
\end{equation}
Substituting this in the relation $x^2 * y=\beta$, by a straight-forward calculation we get $y_0=\beta.$
Thus from the relation $x * y^2=\gamma$, we then get
\begin{equation}\label{y12}
y_1 y_2-y_2 y_1=\gamma.
\end{equation}
Substituting this further in $y^3=\delta$ leaves
\begin{equation}\label{y21^3}
y_2^3-y_1^3=\delta.
\end{equation}
Therefore, $C_\Phi$ is an $F$-algebra generated by $x, y_1, y_2$ subject to the relations $x^3=\alpha$, Equation \eqref{y012},  where $y_0=\beta$, and Equations \eqref{y12}, \eqref{y21^3}. We shall see in particular that, unless $f$ is diagonal, $C_\Phi$ is Azumaya.

Let $w=\beta y_2+\gamma x+y_1^2$. It is a straight-forward calculation to see that $w$, $y_1^3$ and $y_2^3$ commute with $x$, $y_1$ and $y_2$, and therefore they are central in $C_\Phi$. Consider the following affine curve
$$E_\Delta : s^2=r^3+\Delta,$$
where $\Delta=-\gamma^3 \alpha+\gamma^2 \beta^2-\beta^3\delta+ \beta^6.$
We next show that in the case of $\beta \neq 0$, $C_\Phi$ is Azumaya of rank 9 and its center is isomorphic to the coordinate ring of $E_\Delta$.

\begin{lem}
If $\beta \neq 0$ then the subalgebra $F[w,y_1^3]$ of the center of $C_\Phi$ is isomorphic to the coordinate ring $F[r,s]$ of the affine curve $E_\Delta$. In particular it is an integral domain.
\end{lem}

\begin{proof}
If $\beta \neq 0$, then $y_2=\beta^{-1}(w-\gamma x-y_1^2)$ and substituting it in Equation \eqref{y21^3} yields $\beta^{-3} (w^3-\gamma^3 \alpha-y_1^6+\gamma^2 \beta^2)-y_1^3=\delta,$ or equivalently,
\begin{equation*}
w^3+\Delta=(y_1^3 -\beta^3)^2.
\end{equation*}
Consequently $F[w,y_1^3]$ is the $F$-subalgebra generated by $w$ and $y_1^3$ subject only to the relation in the equation above. Thus the map defined by  sending $r, s$ to $w, y_1^3-\beta^3$ clearly gives an $F$-algebra isomorphism.
\end{proof}

Note that $E_\Delta$ is smooth (and then an affine elliptic curve) if and only if its discriminant is nonzero or $\Delta\neq 0$.
In this case, its coordinate ring is a Dedekind domain. In the following, for any integral domain $R$, $q(R)$ stands for its quotient field.

\begin{thm}\label{betaneq}
If $\beta \neq 0$ then
\begin{enumerate}
\item[(1)] $C_\Phi$ is Azumaya of rank 9.
\item[(2)] The center of $C_\Phi$ is the subalgebra $F[w,y_1^3]$, and it is isomorphic to the coordinate ring of $E_\Delta$.
\item[(3)] There is a one-to-one correspondence between the Galois orbits of $\bar{F}$-rational points on $E_\Delta$ and the simple homomorphic images of $C_\Phi$, taking each Galois orbit containing $(r_0,s_0)$ to the degree 3 cyclic algebra $[\alpha\beta^{-3} (s_0+\beta^3),\alpha)_{3,F[r_0,s_0]}$.
\end{enumerate}
\end{thm}

\begin{proof}
In this case, $y_2=\beta^{-1}(w-\gamma x-y_1^2)$. Let $z=\beta^{-1} x y_1$. It is a straight-forward calculation to see that $x z-z x=x$ and $z^3-z=\alpha\beta^{-3} y_1^3$. Consequently, in $C_\Phi \otimes_{F[w,y_1^3]} q(F[w,y_1^3])$, $x$ and $z$ generate over $q(F[w,y_1^3])$ a cyclic algebra of degree $3$ in which $x$ is $3$-central and $z$ is Artin-Schreier. The subalgebra $q(F[w,y_1^3])[x,z]$ in fact contains all the generators of $C_\Phi$, and therefore $q(F[w,y_1^3])[x,z]=C_\Phi \otimes q(F[w,y_1^3])$. In particular, the center of  $C_\Phi \otimes q(F[w,y_1^3])$ is $q(F[w,y_1^3])$, and hence the center of $C_\Phi$  is $F[w,y_1^3]$, which is isomorphic to the coordinate ring $F[r,s]$ of the affine curve $E_\Delta$ by the Lemma above. Identifying $F[w,y_1^3]$ with $F[r,s]$, we have $r=w$ and $s=y_1^3-\beta^3$.

Let there be a simple homomorphic image $A$ of $C_\Phi$. Let $r_0, s_0, x'$ and $y_1'$ be the images in $A$ of $r, s, x$ and $y_1$, respectively. In particular $x'^3=\alpha$ and $y_1'^3=s_0+\beta^3$. Furthermore, $A$ is generated by $x'$ and $z'=\beta^{-1} x' y_1'$ over $F[r_0,s_0]$, where $z'^3-z'=\alpha\beta^{-3} (s_0+\beta^3)$. These two elements satisfy $x' z'-z' x'=x'$, and therefore $A$ is a cyclic algebra of degree 3 over $F[r_0,s_0]$ in which $x'$ is $3$-central and $z'$ is Artin-Schreier. Hence $A$ has the symbol presentation  $[z'^3-z',x'^3)_{3,F[r_0,s_0]}=[\alpha\beta^{-3} (s_0+\beta^3),\alpha)_{3,F[r_0,s_0]}$. In particular, this implies that $C_\Phi$ is Azumaya of rank 9. Consequently, the simple homomorphic images of $C_\Phi$ are determined by the maps taking $F[r,s]$ to $F[r_0,s_0]$ for $\bar{F}$-rational points $(r_0,s_0)$ on the curve $E_\Delta$, whose formula is given as above, and this provides a one-to-one correspondence between the Galois orbits of the $\bar{F}$-rational points on $E_\Delta$ and the simple homomorphic images of $C_\Phi$.
\end{proof}

In case $\beta=0$, $\gamma \neq 0$ and furthermore $\delta\neq 0$, we can simply switch the roles of $x$ and $y$ and get a similar result to Theorem \ref{betaneq}. What remains is the case of $\beta=\gamma=0$.

\begin{thm}
If $\beta=\gamma=0$ then
\begin{enumerate}
\item[(1)] The center of $C_\Phi$ is the polynomial ring $F[y_1]$.
\item[(2)] The algebra $C_\Phi[y_1^{-1}]$ is Azumaya of rank 9 with the Laurent polynomial ring $F[y_1,y_1^{-1}]$ as its center.
\item[(3)] There is a one-to-one correspondence between the Galois orbits of $\bar{F}^\times$ and the simple homomorphic images of $C_\Phi[y_1^{-1}]$, taking each Galois orbit containing $s_0 \in \bar{F}^\times$ to $[\alpha (s_0^3+\delta) s_0^{-3},\alpha)_{3,F[s_0]}$
\item[(4)] The algebra $C_\Phi$ is not Azumaya.
\end{enumerate}
\end{thm}

\begin{proof}
In this case, the algebra $C_\Phi$ is an $F$-algebra generated by $x,y_1, y_2$ subject to the relations $x^3=\alpha$, $[x, y_1]=[y_2, y_1]=0$, $x y_2-y_2 x=y_1$ and $y_2^3-y_1^3=\delta$. Therefore $y_1$ is central in $C_\Phi$ and it generates over $F$ a free algebra in one indeterminate.

The algebra $C_\Phi \otimes_{F[y_1]} q(F[y_1])$ contains the elements $z=x y_2 y_1^{-1}$ and $x$. By a straight-forward calculation we see that $x z-z x=x$ and $z^3-z=\alpha y_2^3 y_1^{-3}$. Since $y_2^3-y_1^3=\delta$, we obtain $z^3-z=\alpha (\delta+y_1^3) y_1^{-3} \in q(F[y_1])$. Thus the $q(F[y_1])$-subalgebra of $C_\Phi \otimes q(F[y_1])$ generated by $x,z$ is cyclic of degree 3, and since it contains all the generators of $C_\Phi$, we see that $q(F[y_1])[x,z]=C_\Phi \otimes q(F[y_1])$. Therefore the center of $C_\Phi \otimes q(F[y_1])$ is $q(F[y_1])$, and hence the center of $C_\Phi$ is $F[y_1]$.

Let $A$ be a simple homomorphic image of $C_\Phi[y_1^{-1}]$. The image of $y_1$ in $A$ is some element $s_0 \in \bar{F}^\times$. Let $x'$ and $y_2'$ be the images of $x$ and $y_2$ in $A$. Now, $x'$ and $z'=x' y_2' s_0^{-1}$ generate a cyclic $F[s_0]$-subalgebra of degree 3, and since they also generate $A$ over $F[s_0]$, we conclude that $A$ is a cyclic algebra over $F[s_0]$ of degree $3$ with the symbol presentation $[z'^3-z',x'^3)_{3,F[s_0]}=[\alpha (s_0^3+\delta) s_0^{-3},\alpha)_{3,F[s_0]}$. Therefore $C_\Phi[y_1^{-1}]$ is Azumaya of rank 9 and the statement (3) follows.

The algebra $C_\Phi$ is not Azumaya, however, because there is one image that is obtained by sending $y_1$ to $0$, namely the commutative $F$-algebra generated by the images $\bar{x}, \bar{y_2}$ of $x$, $y_2$, satisfying $\bar{x}^3=\alpha, \bar{y_2}^3=\delta$.
\end{proof}

\subsection{$e \neq 0$}

By changing the variable $X$ with $X'=e X$, we could assume that $e=1$ in the first place.
Now, by choosing the new pair of variables $X'=X+Y$ and $Y'=X-Y$, we then may assume the polynomial $\Phi$ is of the form $$\Phi(Z,X,Y)=Z^3-(X^2-Y^2) Z-(\alpha X^3+\beta X^2 Y+\gamma X Y^2+\delta Y^3).$$
The algebra $C_\Phi$ thus in this case is
$$F\langle x,y\colon  x^3-x=\alpha, y^3+y=\delta, x^2 * y-y=\beta, x * y^2+x=\gamma\rangle.$$

The element $x$ is Artin-Schreier.
According to Lemma \ref{decomcharp}, $y=y_0+y_1+y_2$ such that $y_k x-x y_k=k y_k$ for $k=0,1,2$.
Substituting that in $x^2 * y-y=\beta$ leaves $y_0=-\beta$.
From $x * y^2+x=\gamma$ we then obtain $y_1 y_2-y_2 y_1+x=\gamma$.
Furthermore, a straight-forward calculation shows that $y^3+y=\delta$ becomes
\begin{equation}\label{y12^3}
y_1^3+y_2^3=\delta+\beta^3+\beta.
\end{equation}
One can check that $w=y_2 y_1-x^2+(1-\gamma) x$, $y_1^3$ and $y_2^3$ are central in $C_\Phi$.

\begin{lem}
The subalgebra $F[w,y_1^3]$ of the center of $C_\Phi$ is isomorphic to the coordinate ring of the affine curve $$E : s^2=r^3+r^2-(\gamma^2+\gamma) r-\alpha^2-\alpha \gamma^3+\alpha \gamma+(\delta+\beta^3+\beta)^2.$$ In particular it is an integral domain.
\end{lem}

\begin{proof}
A straight-forward calculation shows that
\begin{equation*}
\begin{aligned}
w^3 & =y_2^3 y_1^3+w^2+(\gamma^2+\gamma) w-\alpha^2-\alpha \gamma^3+\alpha \gamma\\
 & =(\delta+\beta^3+\beta) y_1^3-y_1^6+w^2+(\gamma^2+\gamma) w-\alpha^2-\alpha \gamma^3+\alpha \gamma.
\end{aligned}
\end{equation*}
Thus $F[w, y_1^3]$ is the algebra over $F$ generated by $w$ and $y_1^3$ subject only to the relation in the equation above. The map defined by sending $r, s$ to $-w, y_1^3+(\delta+\beta^3+\beta)$ then gives an $F$-algebra isomorphism.
\end{proof}


\begin{thm}
Assuming that $\delta+\beta^3+\beta \neq 0$,
\begin{enumerate}
\item[(1)] The algebra $C_\Phi$ is Azumaya of rank 9.
\item[(2)] The center of $C_\Phi$ is the subalgebra $F[w, y_1^3]$, which is isomorphic to the coordinate ring of $E$.
\item[(3)] There is a one-to-one correspondence between the Galois orbits of $\bar{F}$-rational points on $E$ and the simple homomorphic images of $C_\Phi$, taking each Galois orbit containing point $(r_0,s_0)$ to the algebra $[\alpha,s_0-(\delta+\beta^3+\beta))_{3,F[r_0,s_0]}$ if $s_0 \neq \delta+\beta^3+\beta$, and to $[-\alpha,\delta+\beta^3+\beta)_{3,F[r_0]}$ if $s_0=\delta+\beta^3+\beta$.
\end{enumerate}
\end{thm}

\begin{proof}
The element $y_1$ is invertible in $C_\Phi \otimes_{F[w,y_1^3]} q(F[w,y_1^3])$, and so inside this algebra $y_2=(w+x^2-(1-\gamma)x) y_1^{-1}$. Thus $C_\Phi \otimes q(F[w,y_1^3])$ is generated over $q(F[w,y_1^3])$ by $x$ and $y_1$.
Since $x$ is Artin-Schreier, $y_1^3$ is central and $y_1 x-x y_1=y_1$, the algebra $C_\Phi \otimes q(F[w,y_1^3])$ is the cyclic algebra $[\alpha,y_1^3)_{3,q(F[w,y_1^3])}$.
Thus the center of $C_\Phi$ being the intersection of $C_\Phi$ and the center of $C_\Phi \otimes q(F[w,y_1^3])$ is $F[w,y_1^3]$.

Every homomorphism from $C_\Phi$ to a simple algebra $A$ takes $F[w,y_1^3]$ to a field $F[r_0,s_0]$ where $(r_0,s_0)$ is an $\bar{F}$-rational point on the affine curve $E$ and $y_1^3$ is sent to $s_0-(\delta+\beta^3+\beta)$ by the lemma above. If $s_0\neq\delta+\beta^3+\beta$ then $A$ is generated by the images $x', y_1'$ of $x, y_1$ such that $A$ is the cyclic algebra $[x'^3-x',y_1'^3)_{3,F[r_0,s_0]}=[\alpha,s_0-(\delta+\beta^3+\beta))_{3,F[r_0,s_0]}$.
If $s_0=\delta+\beta^3+\beta$ then $y_1^3$ is sent to $0$ and hence $y_2$ is sent to the invertible element $s_0=\delta+\beta^3+\beta$ by Equation \eqref{y12^3}. This means that $A$ is generated by the images $x', y_2'$ of $-x$, $y_2$, satisfying $y_2' x'-x' y_2'=y_2'$. Thus $A$ is the cyclic algebra $[x'^3-x',y_2'^3)_{3,F[r_0]}=[-\alpha,\delta+\beta^3+\beta)_{3,F[r_0]}$. In particular, it implies that $C_\Phi$ is Azumaya of rank 9 and the statement (3) follows.
\end{proof}

\begin{rem}
If $\delta+\beta^3+\beta=0$ then for similar arguments as in the last proof, $C_\Phi[y_1^{-3}]$ is Azumaya of rank 9, and there is a one-to-one correspondence between its simple homomorphic images and the Galois orbits of the $\bar{F}$-rational points $(r_0,s_0)$ on $E$ with $s_0 \neq 0$, taking each such Galois orbit to the algebra $[\alpha,s_0)_{3,F[r_0,s_0]}$.

In this case, the algebra $C_\Phi$ is not necessarily Azumaya, for instance if furthermore  $\gamma^3-\gamma-\alpha=0$ then $F$ is a simple homomorphic image of $C_\Phi$, and then $C_\Phi$ is definitely not Azumaya.
\end{rem}

\section{Algebraically closed and characteristic zero}\label{Secchar0}

In this section we study representations of the generalized Clifford algebra of the monic polynomial $\Phi(Z,X_1,\dots,X_n)=Z^d-\sum_{k=1}^d f_k(X_1,\dots,X_n) Z^{d-k}$ where each $f_k$ is a homogeneous polynomial over the field $F$ of degree $k$.

Recall that an $m$-dimensional representation of an $F$-algebra $A$ is an $F$-algebra homomorphism
$\phi\colon A\rightarrow M_m (K)$ where $K$ is an extension field of $F$. Such a representation is said to be \emph{irreducible} if $\phi(A)K = M_m (K)$. The algebra $\phi (A)$ is sometimes referred to as a representation of $A$.
The following result is a generalization of \cite[Proposition 2.13]{Kuo} with a similar proof.

\begin{prop}\label{greater}
If $\Phi$ is absolutely irreducible, then every representation of the generalized Clifford algebra $C_{\Phi}$ has dimension at least d.
\end{prop}
\begin{proof}
Let $\phi\colon C_{\Phi}\rightarrow M_m(K)$ be a representation. For each $j=1,\dots, n$, let $A_j$ be the image of $x_j$ under $\phi$. Then the element $X_1 A_1 +\cdots +X_n A_n$ in $M_m (K(X_1, \dots, X_n))$ satisfies $\Phi$ regarded as a polynomial in $Z$ with coefficients in $K(X_1, \dots, X_n)$. Note that the polynomial $\Phi$ in $Z$ is irreducible over the field $K(X_1, \dots, X_n)$; otherwise, $\Phi$ as a polynomial in $K[Z, X_1, \dots, X_{n}]$ would be reducible over $K$, a contradiction. Thus $\Phi$ is the minimal polynomial of $X_1 A_1 +\cdots +X_n A_n$ and hence $m$ is at least the degree of $\Phi$.
\end{proof}

From now on assume that $\Phi$ is absolutely irreducible. We have the following result analogous to \cite[Corollary 1.2]{HT}.

\begin{cor}\label{d-repn}
Every $d$-dimensional representation of $C_{\Phi}$ is Azumaya of rank $d^2$. Moreover the representation is irreducible and the kernel of the representation is a prime ideal of  $C_{\Phi}$.
\end{cor}
\begin{proof}
Let $\phi\colon C_{\Phi}\rightarrow M_d(K)$ be a $d$-dimensional representation. Then $\phi( C_{\Phi})$ satisfies all $d\times d$ multilinear identities. Assume $A$ is a simple homomorphic image of $\phi( C_{\Phi})$ satisfying some $l\times l$ multilinear identity for $l<d$. By the Artin-Procesi theorem, $l$ is at least the degree of $A$, say $r$. Let $L$ be a splitting field of $A$. Then $A\otimes_{Z(A)} L\cong M_r (L)$, where $Z(A)$ denotes the center of $A$. It follows that $ C_{\Phi}$ has a representation of dimension $r<d$, a contradiction to Proposition \ref{greater}. Thus no simple homomorphic image of  $\phi(C_{\Phi})$ satisfies any $l\times l$ multilinear identity, where $l<d$. Therefore by the Artin-Procesi theorem, $\phi( C_{\Phi})$ is Azumaya of rank $d^2$ over its center $Z=Z(\phi( C_{\Phi}))$. It follows that the $K$-subalgebra $\phi( C_{\Phi})K$ of $M_d (K)$ is Azumaya of rank $d^2$ over its center $ZK$. Thus $\phi( C_{\Phi})K = M_d (K)$; that  is, $\phi$ is irreducible. In particular, $Z\subseteq K$. Now, let $I$ and $J$ be two ideals of $\phi( C_{\Phi})$ such that $IJ=0$. Then $(I\cap Z)(J\cap Z)=0$ in $Z$, so $I\cap Z =0$ or $J\cap Z =0$. Since $\phi( C_{\Phi})$ is Azumaya, we then have $I=(I\cap Z)\phi( C_{\Phi})=0$ or $J=(J\cap Z)\phi( C_{\Phi})=0$. Thus $\phi( C_{\Phi})$ is a prime ring and hence the kernel of $\phi$ is a prime ideal of $C_{\Phi}$.
\end{proof}

Here comes the question whether there exists any
$d$-dimensional representation of  $C_{\Phi}$. From now on, the
field $F$ is assumed to be algebraically closed of characteristic 0,
and for an $m$-dimensional representation of $C_{\Phi}$, we mean an $F$-algebra
homomorphism $C_{\Phi}\rightarrow M_m(F)$.

Let $Y_{\Phi}$ be the hypersurface in $\mathbb{P}^n$ given by the equation \[Y_{\Phi}\colon\Phi(Z,X_1,\dots,X_n)=0.\] Let $R=F[X_1,\dots,X_n]$ equipped with a grading such that $\textrm{deg}(X_j)=1$. Let $S_{\Phi}$ be the homogeneous coordinate ring of $Y_{\Phi}$; that is,
\[S_{\Phi} = R[Z]/(\Phi(Z,X_1,\dots,X_n)).\] The gradation on $R$ extends naturally to $S_{\Phi}$.

The following result is an analogue of \cite[Proposition 1 in Section 1]{Van}. We first recall that two representations $\phi, \phi'\colon C_{\Phi}\rightarrow M_m (F)$ are equivalent if there exists an invertible matrix $Q\in  M_m (F)$ such that $\phi' (x_j)= Q \phi(x_j) Q^{-1}$, $j=1,\dots ,n$.

\begin{prop}\label{corresp1}
There is a one-to-one correspondence between
\begin{enumerate}
\item[(1)] Equivalence classes of $m$-dimensional representations of  $C_{\Phi}$.
\item[(2)] Graded isomorphism classes of graded $S_{\Phi}$-modules that are isomorphic to $R^m$ as $R$-modules.
\end{enumerate}
\end{prop}
\begin{proof}
Suppose that $\phi\colon C_{\Phi}\rightarrow M_m (F)$ is a representation sending $x_j$ to $A_j$, $j=1,\dots, n$. Define the graded $F$-algebra homomorphism $R[Z]\rightarrow M_m (R)$ by sending $X_j$ to $X_j I_m$, $j=1,\dots, n$ and $Z$ to $X_1 A_1 +\cdots+ X_n A_n$. Then this induces a graded homomorphism $\psi\colon S_{\Phi}\rightarrow M_m (R)$, which clearly is compatible with the $R$-module structure. Via the map $\psi$, $R^m$ forms a graded $S_{\Phi}$-module in the usual way. Conversely, every graded homomorphism from $S_{\Phi}$ into $M_m (R)$ which is compatible with the $R$-module structure gives rise to a representation of $C_{\Phi}$. More explicitly, let $\psi\colon S_{\Phi}\rightarrow M_m (R)$ be such a graded homomorphism. Because of the property of compatibility with the $R$-module structure, we have $\psi(X_j)=X_j I_m$, $j=1,\dots, n$. Let $B=(p_{rs}) \in M_m (R)$ be the image of $Z$ under $\psi$. It follows from $\Phi(B,X_1 I_m, \dots, X_n I_m)=0$ that each $p_{rs}$ is a homogeneous polynomial of degree 1 in $R$. For each $j=1,\dots, n$, let $A_j$ be the matrix in $M_m (F)$ such that the $rs$ entry is the coefficient of $X_j$ in $p_{rs}$, where $1\leq r,s \leq m$.  Then $\phi\colon C_{\Phi}\rightarrow M_m (F)$ sending $x_j$ to $A_j$, $j=1,\dots, n$, is a representation.

Now, suppose that $\phi, \phi' \colon C_{\Phi}\rightarrow M_m (F)$ sending $x_j$ to $A_j$, $A'_j$, respectively, are two equivalent representations; that is, there exists an invertible matrix $Q\in  M_m (F)$ such that $A'_j= Q A_j Q^{-1}$, $j=1,\dots ,n$. Let $\psi, \psi'\colon S_{\Phi}\rightarrow M_m (R)$ be the corresponding graded homomorphisms. Then $\psi'(X_j)=Q \psi(X_j) Q^{-1}$, $j=1,\dots, n+1$, where $X_{n+1}=Z$. Let $P=P'=R^m $ be the corresponding graded $S_{\Phi}$ modules via $\psi, \psi'$, respectively. The map $P\rightarrow P'$ defined by sending each $m$-tuple $\tilde{a}$ to $Q\tilde{a}$ is an isomorphism of graded $S_{\Phi}$-modules. Conversely, suppose that $P=P'=R^m$ form two graded $S_{\Phi}$-modules via the graded homomorphisms $\psi,\psi'\colon S_{\Phi}\rightarrow M_m (R)$, respectively, and that $\Psi\colon P\rightarrow P'$ is a graded $S_{\Phi}$-module isomorphism. Since $\Psi$ is of degree 0, it acts as the multiplication on the left by an invertible matrix $Q$ in $M_m (F)$. Furthermore, for each $\tilde{a}\in P$, since $\Psi(X_j\cdot \tilde{a})= X_j\cdot \Psi(\tilde{a})$, it follows that $Q \psi (X_j)= \psi' (X_j) Q$, $j=1,\dots, n+1$. Let $\phi, \phi'\colon C_{\Phi}\rightarrow M_m (F)$ be the two representations induced by $\psi, \psi'$  as described above. Then $\psi(Z)=\Sigma_{j=1}^n X_j\phi(x_j)$ and $\psi'(Z)=\Sigma_{j=1}^n X_j\phi'(x_j)$. Thus $\Sigma_{j=1}^n X_j Q\phi(x_j)=\Sigma_{j=1}^n X_j \phi'(x_j)Q$ and hence $\phi'(x_j)=Q\phi(x_j)Q^{-1}$ for each $j=1,\dots, n$; that is, $\phi$ and $\phi'$ are equivalent.
\end{proof}

From this proposition we derive the following corollary, which answers the second part of Question 5.2 in \cite{Pappacena} in a better way than Proposition \ref{greater}.

\begin{cor}\label{divide}
The dimension of every representation of $C_{\Phi}$ is divisible by $d$.
\end{cor}
\begin{proof}
Notice that $\Phi$ is assumed to be irreducible. In particular, $f_d(X_1, \dots, X_n)$ in the definition of $\Phi$ is nonzero. Thus the $S_{\Phi}$-modules given as in Proposition \ref{corresp1} are torsion-free, and hence must have $R$-ranks divisible by $d$.
\end{proof}

Since there is a one-to-one correspondence between graded
$S_{\Phi}$-modules and $\mathcal{O}_{Y_{\Phi}}$-modules, it follows that there
is a one-to-one correspondence between graded isomorphism classes of
graded $S_{\Phi}$-modules as in Proposition \ref{corresp1}(2) and
isomorphism classes of $\mathcal{O}_{Y_{\Phi}}$-modules $\mathcal{M}$
such that $\pi_{*}\mathcal{M}\cong
\mathcal{O}_{\mathbb{P}^{n-1}}^{m}$, where $\pi\colon Y_{\Phi}\rightarrow
\mathbb{P}^{n-1}$ is the projection which forgets the variable
$Z$. From now on assume that the hypersurface $Y_{\Phi}$ is
nonsingular, which implies that $\mathcal{M}$ is locally free.
Therefore, we conclude the following result from Proposition \ref{corresp1} and Corollary
\ref{divide}.

\begin{prop}\label{corresp2}
There is a one-to-one correspondence between
\begin{enumerate}
\item[(1)] Equivalence classes of $dr$-dimensional representations of  $C_{\Phi}$.
\item[(2)] Isomorphism classes of vector bundles $\mathcal{M}$ of rank $r$ on $Y_{\Phi}$ such that $\pi_{*}\mathcal{M}\cong \mathcal{O}_{\mathbb{P}^{n-1}}^{dr}$.
\end{enumerate}
\end{prop}

Such a vector bundle $\mathcal{M}$ on $Y_{\Phi}$ in the proposition above
is called Ulrich (cf. \cite[Definition 2.6]{CKM}). With Proposition
\ref{corresp2} at hand, one can check that the remaining arguments
in \cite{Van}  and \cite{CKM} for the Clifford algebra of a
form can go through as well in this case of the generalized Clifford algebra $C_{\Phi}$. Thus we list the following results without proofs.

\begin{prop}(cf. \cite[Corollary 2.19]{CKM})
In any case, the generalized Clifford algebra $C_{\Phi}$ admits an irreducible representation.
\end{prop}

This gives a positive answer for the first part of
Question 5.2 in \cite{Pappacena}.

\begin{prop}(cf. \cite[Proposition 2]{Van})
In the case of $n=2$ and $d=3$, all irreducible representations of $C_{\Phi}$ have dimension 3, and are in one-to-one correspondence with the points of $Y_{\Phi}\setminus O$, where $O\in Y_{\Phi}$.
\end{prop}

This result in particular is in concert with what we have derived in Section 3. As for the case of $n=2$ and $d>3$ and the case of $n=3$ and $d=3$, the following two propositions show that $C_{\Phi}$ admits irreducible representations of arbitrary high dimension. In particular, $C_{\Phi}$ is not finitely generated over its center.

\begin{prop}(cf. \cite[Theorem 4]{Van})
In the case of $n=2$ and $d>3$, there are irreducible representations of $C_{\Phi}$ of dimension $dr$ for arbitrary $r\geq 1$.
\end{prop}

\begin{prop}(cf. \cite[Corollaries 3.2 and 3.6]{CKM})
In the case of $n=3$ and $d=3$, there exist $3r$-dimensional irreducible representations of $C_{\Phi}$ for arbitrary $r\geq 1$. Moreover, for each $r$ there are finitely many equivalence classes of such representations and when $r=1$ there are exactly 72.
\end{prop}

\section*{Acknowledgement}
The first author is currently supported by Wallonie-Bruxelles International.
The second author was partially supported by the National Science Council of Taiwan under grant NSC 102-2115-M-005-002.

\section*{References}
\bibliographystyle{amsalpha}

\end{document}